\documentclass[a4paper,11pt]{amsart} %automatically loads amsmath and amsthm
     \usepackage[utf8x]{inputenc}
\usepackage {mathrsfs} %fuer mathscr
\usepackage{hyperref}
\usepackage[all]{xy} % fuer diagramme
\usepackage[lite]{amsrefs}
\usepackage[left=3cm,top=3cm,right=3cm,bottom=3cm]{geometry}

\newtheorem{thm}{Theorem}[section]
\newtheorem*{thm*}{Theorem}
\newtheorem{qu}{Question}
\newtheorem{lem}[thm]{Lemma}
\newtheorem{cor}[thm]{Corollary}
\newtheorem*{cor*}{Corollary}
\newtheorem{prop}[thm]{Proposition}
\theoremstyle{definition}
\newtheorem*{defn}{Definition}
\newtheorem{rem}[thm]{Remark}
\theoremstyle{remark}

\newtheorem{examp}{Example}
\newtheorem*{nota}{Notation}
\hypersetup{
  colorlinks   = true, %Colours links instead of ugly boxes
  urlcolor     = blue, %Colour for external hyperlinks
  linkcolor    = blue, %Colour of internal links
  citecolor   = red %Colour of citations
}
\DeclareMathOperator{\R}{{\mathbb R}}
  \DeclareMathOperator{\Z}{{\mathbb Z}}
  \DeclareMathOperator{\N}{\mathbb N}
  \DeclareMathOperator{\Ss}{{\mathscr{U}}}
  
  \renewcommand{\SS}{{\mathscr U}}
  \DeclareMathOperator{\Pp}{{\mathscr P}}
  \DeclareMathOperator{\PP}{{\mathscr P}}
  \DeclareMathOperator{\Aut}{Aut}
  \DeclareMathOperator{\Sl}{SL}
  \DeclareMathOperator{\Perm}{Perm}
 \DeclareMathOperator{\conv}{conv}
 \DeclareMathOperator{\id}{id}

  \DeclareMathOperator*{\dcup}{\sqcup}
\DeclareMathOperator*{\limto}{\longrightarrow}
  \renewcommand{\phi}{\varphi}
  \renewcommand{\epsilon}{\varepsilon}
  
  \DeclareMathOperator{\diam}{diam}

  \newcommand{\frei}{\parbox{1cm}{}}
\title{Amenable groups and a geometric view on unitarisability}
\author{Peter Schlicht}
\address{P.S., EPFL, SB MATHGEOM EGG, MA C3 584, Station 8, CH-1015 Lausanne }
\email{Peter.Schlicht@epfl.ch}
\date{\today}
\begin{document}
\begin{abstract}
  We investigate unitarisability of groups $\Gamma$ by looking at actions of $\Gamma$ on the cone of positive invertible operators of a Hilbert space. This way, we can reprove results previously attained by Gilles Pisier, in a rather intuitive and geometric fashion.\\[0.3cm]
By constructing barycenters of finite sets in a more general class of geodesic metric spaces, we prove a fixed point theorem for actions of amenable groups by isometries. In particular, we give geometric proofs for the facts that extensions of unitarisable groups by amenable groups are unitarisable and virtual unitarisability implies unitarisability and calculate some corresponding universal constants.
\end{abstract}
\maketitle
\tableofcontents %Inhaltsverzeichnis drucken
\section{Introduction}
A representation $\pi$ of a group $\Gamma$ on a Hilbert space $H$ is called \textbf{unitarisable}, if there is an operator $S$ in the algebra $B(H)$ of bounded operators on $H$, such that $S\pi(\gamma)S^{-1}$ is a unitary for every $\gamma\in\Gamma$. Such representations are \textbf{uniformly bounded} by $\Vert S\Vert\cdot\Vert S^{-1}\Vert$ (i.e. $\Vert\pi(\gamma)\Vert$ has a uniform bound over all of $\Gamma$).\par
The group $\Gamma$ is then called a $\textbf{unitarisable group}$, iff every uniformly bounded representation is unitarisable. It was found in 1947 (\cite{sznagy}), that the group $\Z$ of integers is unitarisable. Later, this was generalized to amenable groups (independently \cite{day}, \cite{dix}, \cite{naka}), which led Dixmier to pose the following question:
\begin{qu}[Dixmier]
Is every unitarisable group amenable?
\end{qu}
First examples of non-unitarisable groups were found in 1955 by Ehrenpreis and Mautner (\cite{ehrenp}), who showed, that $\Sl_2(\R)$ is not unitarisable. Later (e.g. \cite{pytsz}), explicit examples of uniformly bounded representations of the free group on two generators were constructed.\par
First examples of non-unitarisable groups not containing non-abelian free subgroups were given by Epstein and Monod in \cite{epm} using groups constructed by Osin (\cite{osin}). For a more detailed survey on this subject, the reader may be referred to \cite{pis3} or \cite{oza}.\par
Given a representation $\pi$ as above, one can define an action of $\Gamma$ on the cone $\Pp(H)$ of positive invertible operators on $H$, which has a fixed point if and only if $\pi$ is unitarisable.\par
Moreover, this action is isometric with respect to the Thompson metric defined by $d(x,y):=\Vert\ln x^{-\frac12}yx^{-\frac12}\Vert$, where $\Vert\cdot\Vert$ denotes the operator norm.\par
In the first part of this article, by investigating this metric we give a geometric and simple proof to the following theorem, which was previously attained algebraically by Gilles Pisier in \cite{pis}. Here, by $\vert\pi\vert$, we denote the smallest uniform bound for a uniformly bounded representation $\pi$.
\begin{thm}\label{eins}
For a unitarisable group $\Gamma$, there are constants $K(\Gamma),\alpha(\Gamma)>0$ such that for every uniformly bounded $\pi$, there is a bounded operator $S$ unitarising $\pi$ with $\Vert S\Vert\cdot\Vert S^{-1}\Vert\leq K(\Gamma)\cdot\vert\pi\vert^{\alpha(\Gamma)}$.
\end{thm}
This theorem translates nicely into the metric setup of actions on $\Pp(H)$ coming from uniformly bounded representations (Corollary \ref{einsgeom}) which draws our attention to $\Pp(H)$ as a metric space. In the second part, we then prove some topological and geometric facts about $\left(\Pp(H),d\right)$. In particular, we show that the metric topology coincides with the restriction of the toplogy coming from the operator norm (Theorem \ref{sametop}) and construct midpoint sets for bounded sets.\par
Finally, we introduce the concept of a GCB-space, which are complete geodesic spaces generalizing complete CAT(0)-spaces and the space $(\Pp(H),d)$. In such spaces, we construct barycenters for finite sets and deduce a fixed point theorem for actions of amenable groups.\par
As consequences of this theorem, we show, that virtual unitarisability is equivalent to unitarisability and that extensions of unitarisable groups by amenable groups are unitarisable. Both these facts seem not too deep but have apparently not yet appeared in the literature. Depending on the universal constants $\alpha(\Gamma')$ and $K(\Gamma')$ in Theorem \ref{eins}, we calculate the corresponding constants for unitarisable groups $\Gamma$, which are extensions of a unitarisable group $\Gamma'$ by an amenable group or contain $\Gamma'$ as a finite index unitarisable subgroup (Corallary \ref{virtunit} and Theorem \ref{extend}). This way, we show that $\Gamma$ will not be "less amenable", than $\Gamma'$ (see Remark \ref{end}).
\section{A geometric approach to unitarisability}
\begin{nota}As in the introduction, $\Gamma$ shall always denote a discrete and countable group and  $\pi$  a uniformly bounded, linear representation of $\Gamma$ on a separable Hilbert space $H$.\par
We denote by $B(H)$ the algebra of bounded operators on $H$ and by $\Pp(H)\subset B(H)$ the cone of positive and invertible operators. We will often omit the $\pi$ in speaking about the images under $\pi$ of $\gamma\in\Gamma$. In particular, $\gamma^*$ shall mean $\pi(\gamma)^*$.\end{nota}
For a uniformly bounded representation $\pi$, $\vert\pi\vert\in\R_{\geq 1}$ stands for the smallest uniform bound, which we will call the \textbf{size of the representation}. By the \textbf{size $s(S)$ of an invertible operator} $S$, we mean the positive number
$s(S):=\Vert S\Vert\cdot\Vert S^{-1}\Vert$. If $\pi$ is unitarisable, we shall denote the set of all unitarisers by $\Ss(\pi)$.
\subsection{About smallest unitarisers and fixed points}
For a given representation $\pi$ of a group $\Gamma$ on a Hilbert space $H$, we can define the following action of $\Gamma$ on $\Pp(H)$:
\begin{align*}\rho_\pi:\Gamma\times\PP(H)\to\PP(H),~(\gamma,P)\mapsto\gamma P\gamma^*
\end{align*}
To keep formulas simple, we will abbreviate the notation by writing $\gamma x$ instead of $\rho(\gamma,x)$ whenever the action is clear from the context. In particular $\Gamma x$ is the orbit of $x$ under $\rho$.\par
The action $\rho_\pi$ has the following property
\begin{lem}
\label{fixedunit}
For every $S\in\SS(\pi)$, $SS^*$ is a fixed point of $\rho_\pi$. Conversely, for any fixed point $T$ of $\rho_\pi$ one has $\sqrt T\in\SS(\pi)$. Moreover, every unitarisable $\pi$ has a positive unitariser.
\end{lem}
\begin{proof}\frei\\
This is a straightforward calculation:
\begin{align*}
S\in\Ss&\Leftrightarrow S^{-1}\pi(\gamma)S\in U(H)~\forall \gamma\in \Gamma\\
&\Leftrightarrow~\id_H=\left(S^{-1}\pi(\gamma)S\right)\left(S^{-1}\pi(\gamma)S\right)^*=
S^{-1}\pi(\gamma)SS^*\pi(\gamma)^*\left(S^{-1}\right)^*&\forall \gamma\in \Gamma\\
&\Rightarrow SS^*=\pi(\gamma)SS^*\pi(\gamma)^*=\rho_\pi(\gamma,SS^*)&\forall \gamma\in \Gamma
\end{align*}
and conversely, for any $\gamma\in\Gamma$
\begin{align*}
&T=\rho_\pi(\gamma,T)&\Rightarrow&\sqrt T^2=T=\pi(\gamma)T\pi(\gamma)^*\\
&&\Rightarrow&\id_H=\sqrt T^{-1}\pi(\gamma)T\pi(\gamma)^*\left(\sqrt T^*\right)^{-1}=\sqrt T^{-1}\pi(\gamma)\sqrt T\left(\sqrt T^{-1}\pi(\gamma)\sqrt T\right)^*\\
&&\Rightarrow&\sqrt T\in\SS(\pi)
\end{align*}
For the second part of the claim, let $S$ unitarise a representation $\pi$. Then, by the above calculation,$\sqrt{SS^*}\in\Pp(H)$ also does.
\end{proof}
The following lemma, which will be helpful in proving Theorem \ref{eins}, states that for any unitarisable representation, we can find a \textbf{smallest unitariser}.
\begin{lem}\label{infgleichmin}
Let $\pi$ be a unitarisable representation. Then $\inf\limits_{S\in\SS(\pi)}s(S)=\min\limits_{S\in\SS(\pi)}s(S)$.\\
Moreover, this smallest unitariser may be chosen to lie in $\Pp(H)$.
\end{lem}
\begin{proof}\frei\\
Using Lemma \ref{fixedunit}, the mapping $S\mapsto\sqrt{SS^*}$ maps elements in $\Ss(\pi)$ to $\Ss(\pi)\cap\Pp(H)$ of same size. Hence it suffices to consider positive unitarisers.\par
Now, by Lemma \ref{fixedunit}, there is a size-squaring bijection
$S\mapsto S^{2}$ between $\Ss(\pi)\cap\Pp(H)$ and the set of fixed points of $\rho_\pi$. \par
Therefore, the claim is equivalent to predicting the existence of some operator $T$ in the convex and norm-closed set $\Pp(H)^\Gamma$ of fixed points for $\rho_\pi$, which minimize the size.\\
Finally, for any $T\in\Pp(H)^\Gamma$ and $\gamma \in \Gamma$ 
\begin{align*}
  \rho_\pi(\gamma ,\lambda T)=\gamma\left(\lambda T\right)\gamma^*=\lambda\cdot\gamma T\gamma^*=\lambda T~\forall\lambda\in\R_+
\end{align*}
showing that $\Pp(H)^\Gamma$ is closed under multiplication with positive scalars, which preserves both, size and positivity.\par
Define $\Pp(H)^\Gamma_1$ to be the set $\Pp(H)^\Gamma\cap\{A:\Vert A\Vert=1\}\subset\Pp(H)^\Gamma$ of fixed points of norm 1. We have reduced the claim to $\inf\left\{s(T),~T\in\Pp(H)^\Gamma_1\right\}=\min\left\{s(T),~T\in\Pp(H)^\Gamma_1\right\}$.\par
As for such $T$, one has 
$s(T)=\Vert T^{-1}\Vert=(\min\sigma(T))^{-1}=(1-\Vert\id_H-T\Vert)^{-1}$, searching a $T\in\Pp(H)^\Gamma_1$ of minimal size means looking for an operator that realizes the linear distance $\kappa$ from $\id_H$ to $\Pp(H)^\Gamma_1$. \par
Let $(T_i)_{i\in\N}$ be a sequence in $\Pp(H)_1^\Gamma$, realizing this distance. Then it is obvious, that $\Pp(H)^\Gamma_1\subset\{A\in B(H):~\Vert A\Vert\leq 1\}$ and this is compact with respect to the weak operator topology. Now, a limit of a weak operator convergent subsequence is easily shown to have the demanded properties.
\end{proof}\pagebreak
\begin{defn}
  For a unitarisable representation $\pi$ of $\Gamma$, some $t\in[0,1]$ and a chosen smallest and positive unitariser $S$ of $\pi$, we define $\pi_t$ by $\pi_t:\gamma \mapsto S^{-t}\pi(\gamma )S^t$.
\end{defn}
\begin{lem}\label{interpol}
Let $\pi$ be unitarisable and $S$ a smallest and positive unitariser. Then $S^{1-t}$ is a smallest unitariser for $\pi_t$.
  \end{lem}
\begin{proof}\frei\\
Obviously, the definition of $\pi_t$ is the same for $S$ replaced by  $\lambda S$ for some positive $\lambda$ and without loss of generality, we may assume $\Vert S\Vert=1$. Assume for contradiction the existence of some $Q\in\SS(\pi_t)$ with $\Vert Q\Vert=1$, such that $\Vert Q^{-1}\Vert=s(Q)\leq s(S^{1-t})=\Vert S^{t-1}\Vert$.\par
Then $S^tQ$ obviously unitarises $\pi_t$ and
\begin{align*}
s(S^tQ)&=\left\Vert S^tQ\right\Vert\cdot\left\Vert\left(S^tQ\right)^{-1}\right\Vert\leq\left\Vert S^t\right\Vert\cdot\left\Vert S^{-t}\right\Vert\cdot\left\Vert Q\right\Vert\cdot\left\Vert Q^{-1}\right\Vert<\left\Vert S^{-1}\right\Vert^t\cdot\left\Vert S^{-(1-t)}\right\Vert=s(S) 
\end{align*}
contradicting $S$ to be the smallest unitariser.
\end{proof}
\begin{cor}\label{similar}
 If $\alpha$ is the size of the smallest unitariser of $\pi$, then the smallest unitariser of $\pi_t$ has size $\alpha^{1-t}$.
\end{cor}
\begin{proof}\frei\\
This is immediate from the previous lemma.
\end{proof}
\subsection{A metric structure on $\Pp(H)$}
The space $\Pp(H)$ carries a metric structure by
\begin{align*}
d(x,y)=\left\Vert\ln\left(x^{-\frac12}yx^{-\frac12}\right)\right\Vert
\end{align*}
This metric is sometimes called the \textbf{Thompson metric}.\par
The set $\{\phi_a:x\mapsto axa^*,a\in B(H)~\textrm{invertible}\}$ is a transitive subgroup of the group of isometries for this metric. In fact, every isometry of $\Pp(H)$ is of the form $x\mapsto ax^\epsilon a^*$ for some (conjugate-)linear $a$ and $\epsilon\in\{\pm1\}$. One may be referred to \cite{moln}, Theorem 2 and  \cite{thomp} for proofs of those facts.\par
Also, for any two points $x,y\in\Pp(H)$, there is a geodesic (i.e. a continuous curve of length $d(x,y)$ connecting $x$ and $y$) between them. They are given by
\begin{align*}
\eta(x,y,\cdot):I\to\Pp(H),~\eta(x,y,t)=x^{\frac12}\left(x^{-\frac12}yx^{-\frac12}\right)^tx^{\frac12}
\end{align*}
Those geodesics can easily be seen to be mapped to one-another by the maps $\phi_a$. In particular, for a representation $\pi$, $\rho_\pi$ is an action of isometries respecting those geodesics: \begin{align}\label{equiv}\rho_\pi(\gamma)\cdot\eta(x,y,t)=\eta(\rho_\pi(\gamma)x,\rho_\pi(\gamma)y,t)~\forall x,y\in\Pp(H),\forall \gamma\in\Gamma,~\forall t\in I\end{align}
See \cite{schlicht} for details.
\begin{thm}[\cite{correcht}]\label{respect}
The metric $d$ on $\Pp(H)$ is convex with respect to $\eta(x,y,\cdot)$:
\begin{align}\label{gammaconv}
d\left(\eta\left(x_1,x_2,t\right),\eta\left(y_1,y_2,t\right)\right)\leq td(x_2,y_2)+(1-t)d(x_1,y_1)~\forall t\in[0,1]~\forall x_i,y_i\in\Pp(H)
\end{align}
In particular, metric balls are metrically convex.
\end{thm}
\begin{rem}
The geodesics given above are not unique as "metric geodesics". The space $\Pp(H)$ can though be given the structure of a Finsler manifold and in this differential geometric set up, they are unique (as self-parallel curves). See \cite{correcht}, for example. This motivates the following definition:
\end{rem}
\begin{defn}
A subset $A\subset X$ of a metric space with chosen geodesics $\eta(x,y,\cdot)$ between any two points $x,y\in X$ is called \textbf{metrically convex} (or just convex, if no other term of convexity applies), if for any $x,y\in A$ one has $\eta(x,y,t)\in A~\forall t\in[0,1]$.\end{defn}

\begin{lem}
 For any action of a group $\Gamma$ on $\Pp(H)$ respecting the geodesics, the fixed point set $\PP(H)^\Gamma$ is convex.
\end{lem}
\begin{proof}\frei\\
  For $x,y\in\PP(H)^\Gamma$ and $\gamma\in\Gamma$ we have $\gamma\cdot\eta(x,y,t))\stackrel{(\ref{equiv})}=\eta(\gamma\cdot x,\gamma\cdot y,t)=\eta(x,y,t)$.
\end{proof}
For a uniformly bounded representation $\pi$ of some $\Gamma$ and any $\gamma\in\Gamma$, one obviously has $\vert\pi\vert^2\geq\Vert\pi(\gamma )\Vert^2=\Vert\pi(\gamma )\pi(\gamma )^*\Vert$ and analogously $\vert\pi\vert^2\geq\left\Vert\left(\pi(\gamma )\pi(\gamma )^*\right)^{-1}\right\Vert$.\par
Hence, the $\Gamma$-orbit $\rho_\pi(\Gamma,\id_H)$ of $\id_H\in\Pp(H)$ under $\rho_\pi$ (and thus its closed convex hull) is bounded:
\begin{align*}d(\rho_\pi(\gamma ,\id_H),\id_H)&=\Vert\ln(\pi(\gamma )\pi(\gamma ^*))\Vert\\&=\max\left\{\ln\left(\left\Vert(\pi(\gamma )\pi(\gamma ^*))\right\Vert\right),\ln\left(\left\Vert(\pi(\gamma )\pi(\gamma ^*))\right\Vert^{-1}\right)\right\}\geq\ln\left(\vert\pi\vert^2\right)\end{align*}
This motivates the following definition.
\begin{defn}
 For a uniformly bounded representation $\pi$ of a group $\Gamma$, one defines
\begin{align*}
 \diam(\pi):=\sup\limits_{\gamma_1,\gamma_2\in \Gamma}d(\rho_\pi(\gamma_2,\id_H),\rho_\pi(\gamma_1,\id_H))=\diam(\rho_\pi(\Gamma,\id_H))
\end{align*}
to be the \textbf{diameter of $\pi$}.\par
Since $\rho_\pi$ is an action of isometries on $\Pp(H)$, this coincides with $\sup\limits_{\gamma \in \Gamma}d(\id_H,\rho_\pi(\gamma ,\id_H))$.
\end{defn}
\begin{lem}\label{diamvert}
  For a uniformly bounded representation $\pi$, one has $\diam(\pi)=2\ln\vert\pi\vert$.
\end{lem}
\begin{proof}\frei\\
  One calculates
\begin{align*}
  \diam(\pi)&=\sup\limits_{\gamma\in \Gamma}d\left(\id_H,\rho_\pi(\gamma,\id_H)\right)=\sup\limits_{\gamma\in \Gamma}\left\Vert\ln\left(\pi(\gamma)\pi(\gamma)^*\right)\right\Vert\\&=\sup\limits_{\gamma\in \Gamma}\max\left\{\ln\left(\left\Vert\pi(\gamma)\pi(\gamma)^*\right\Vert\right),\ln\left(\left\Vert\left(\pi(\gamma)\pi(\gamma)^*\right)^{-1}\right\Vert\right)\right\}\\
&=2\sup\limits_{\gamma\in \Gamma}\ln\left(\left\Vert\pi(\gamma)\right\Vert\right)=2\ln\vert\pi\vert
\end{align*}
\end{proof}
\begin{lem}\label{growing}  For $\pi_t$ as in Lemma \ref{interpol}, $\vert\pi_t\vert\leq\vert\pi\vert^{1-t}~\forall t\in I$.
\end{lem}
\begin{proof}\frei\\
 Using the facts from Theorem \ref{respect} and Lemma \ref{fixedunit} for the metric $d$, one calculates
\begin{align*}
2\ln\vert\pi_t\vert&=\diam(\pi_t)=\sup\limits_{\gamma \in \Gamma}d(\id_H,\pi_t(\gamma )\pi_t(\gamma )^*)=\sup\limits_{\gamma \in \Gamma}d\left(\id_H,S^{-t}\pi(\gamma )S^{2t}\pi(\gamma )^*S^{-t}\right)\\&=\sup\limits_{\gamma\in \Gamma}d\left(S^{2t},\gamma S^{2t}\right)=\sup\limits_{\gamma \in \Gamma}d\left(\eta(\id_H,S^2,t),\gamma \eta\left(\id_H,S^2,t\right)\right)\\
&=\sup\limits_{\gamma \in \Gamma}d\left(\eta(\id_H,S^2,t),\eta\left(\gamma \id_H,\gamma S^2\right),t\right)\\
&\stackrel{(\ref{gammaconv})}\leq\sup\limits_{\gamma \in \Gamma}\left((1-t)d(\id_H,\gamma \id_H)+t d(S^2,\gamma S^2)\right)\\
&=\sup\limits_{\gamma \in \Gamma}(1-t)d\left(\id_H,\gamma \id_H\right)=(1-t)\diam(\pi)=2(1-t)\ln\left(\vert\pi\vert\right)=2\ln\left(\vert\pi\vert^{1-t}\right)
\end{align*}
This obviously implies $\ln\vert\pi_t\vert\leq\ln\left(\vert\pi\vert^{1-t}\right)$ and exponentiating both sides yields the claim.
\end{proof}
\begin{lem}\label{pitcont}
 For $\pi_t$ as in Lemma \ref{interpol}, the function $t\mapsto\vert\pi_t\vert$ is continuous
\end{lem}
\begin{proof}\frei\\
 Above we have seen, that $2\ln\vert\pi_t\vert=\sup\limits_{\gamma \in \Gamma}d\left(\eta(\id_H,S^2,t),\gamma \eta\left(\id_H,S^2,t\right)\right)$.
Now, the claim is proven, if the right hand side depends continuously on $t$.\par
This follows easily from the fact, that the family over which we take the supremum, is uniformly equicontinuous: one uses the following easy consequence of the triangle inequality for arbitrary 4 points $a,b,c,d$ in a metric space:
\begin{align}\label{quadrupel}
 \left\vert d(a,d)-d(b,c)\right\vert\leq d(a,b)+d(c,d)
\end{align}
 Now, for $\epsilon>0$, let $\delta=\frac{\epsilon}{4\Vert\ln S\Vert}$ and choose $t,t'\in[0,1]$ such that $\vert t-t'\vert<\delta$.\par
Then, for arbitrary $\gamma \in \Gamma$, one has (as $\Gamma$ acts by isometries)
\begin{align*}
&\left\vert d\left(\eta(\id_H,S^2,t),\gamma \eta\left(\id_H,S^2,t\right)\right)-d\left(\eta(\id_H,S^2,t'),\gamma \eta\left(\id_H,S^2,t'\right)\right)\right\vert\\
&\hspace{0.5cm}\stackrel{(\ref{quadrupel})}\leq d(\eta(\id_H,S^2,t),\eta(\id_H,S^2,t'))+d(\gamma\eta(\id_H,S^2,t),\gamma \eta(\id_H,S^2,t'))\\
&\hspace{0.5cm}=2d(\eta(\id_H,S^2,t),\eta(\id_H,S^2,t'))=2\left\Vert\ln\left(S^{2(t'-t)}\right)\right\Vert=4\vert t'-t\vert\Vert\ln S\Vert<\epsilon
\end{align*}
\end{proof}
Now, we can prove Theorem \ref{eins}:
\begin{thm*}
 For a unitarisable group $\Gamma$, there are universal constants $K(\Gamma)$ and $\alpha(\Gamma)\in\R_+$ depending only on $\Gamma$, such that for every uniformly bounded representation $\pi$ of $\Gamma$ on some Hilbert space $H$ the following holds
\begin{align*}\exists S\in\Ss(\pi):~s(S)\leq K\cdot\vert\pi\vert^\alpha\end{align*}
\end{thm*}
\begin{proof}\frei\\
 We assume for contradiction that this is not the case.\par
So, choosing $K=\alpha=n\in\N$ yields uniformly bounded representations $\pi_n:=\pi_{n,n}$ of $\Gamma$ on Hilbert spaces $H_n$ with smallest unitarisers $S_n:=S_{n,n}$, such that $s(S_n)> n\vert\pi_n\vert^n$.\par
 In order to find a contradiction, we would like to consider the direct sum of those representations. Of course, this does not have to be uniformly bounded, as the sequence $\left(\vert\pi_n\vert\right)_{n\in\N}$ has no reason to be bounded from above.\par
For a given $\pi_n$ such that $\vert\pi_n\vert>2$ and in the flavour of Lemma \ref{growing}, we define \begin{align*}\pi_{n,t}(\gamma ):=S_n^{-t}\pi_n(\gamma )S_n^{t}\end{align*}
By Lemma \ref{growing} and Lemma \ref{pitcont}, we can then find a $0<t<1$ yielding $2=\vert\pi_{n,t}\vert\leq\vert\pi_n\vert^{1-t}$ and the corresponding smallest unitariser $S_{n,t}=S_n^{1-t}$ of $\pi_{n,t}$ fullfills by Corollary \ref{similar} and Lemma \ref{growing}
\begin{align*}
s(S_{n,t})&=s(S_n)^{1-t}>(n\vert\pi_n\vert^n)^{1-t}\geq n^{1-t}\vert\pi_{n,t}\vert^n> 2^n> n
\end{align*}
As the size of every representation is at least $1$, we also have for all those $\pi_n$ with $\vert\pi_n\vert\leq2$
\begin{align*}
 s(S_n)> n\vert\pi_n\vert^n\geq n
\end{align*}
This way, we get a sequence $(\pi_n:\Gamma\to\Aut(H_n))_{n\in\N}$ of uniformly bounded representations of $\Gamma$, such that for any $n\in\N$ $\vert\pi_n\vert\leq 2$ and $s(S_n)> n$ hold.\par
Now, let $\pi=\bigoplus\limits_{n\in\N}\pi_n$. By taking suprema over all $\gamma \in \Gamma$, we get $\vert\pi\vert=\sup\limits_{n\in\N}\left\vert\pi_n\right\vert\leq2$ and $\pi$ is itself a uniformly bounded representation of $\Gamma$. In particular we find a bounded $S$ unitarising $\pi$. 
% Again, we may assume without loss of generality that $S$ is selfadjoint and (by scaling with $\Vert S^{-1}\Vert^{-1}$), that $\Vert S^{-1}\Vert=1$
\par
But then $S\pi_n(\gamma ) {S^{-1}}_{|_{SH_n}}$ is unitary for every $n\in\N,~\gamma \in \Gamma$. Choosing any unitary equivalence $U:SH_n\to H_n$ we get, that $US_{|_{H_n}}:H_n\to H_n$ unitarises $\pi_n$ and hence,
\begin{align*}s\left(US_{|_{H_n}}\right)\geq s(S_n)> n~\forall n\in\N\end{align*}
This contradicts the boundedness of $S$ as $\left\Vert US_{|_{H_n}}\right\Vert\leq \Vert S\Vert$.
\end{proof}
We now aim at translating the above theorem into our geometric setup. In Lemma \ref{diamvert}, we have already seen that the size of some uniformly bounded $\pi$ corresponds to the diameter of the $\rho_\pi$-orbit of $\id_H$. It will turn out that the size of a smallest unitariser corresponds to the distance of $\id_H$ to the fixed point set of $\rho_\pi$. But unlike the size of an operator, the set of unitarisers or the fixed-point-set $\Pp(H)^\Gamma$, which are closed under scaling with positive real numbers, the ``metric counterpart'' $d(SS^*,\id_H)$ of the size of a smallest positive unitariser is not. The following lemma implies, that one can construct a fixed point coming from a smallest unitariser, which also realizes the distance to $d(\id_H,\Pp(H)^\Gamma)$.
\begin{lem}\label{symspec}%\parbox{1cm}{}\\
 Let $\pi$ be a unitarisable representation of $\Gamma$ and $\rho_\pi$ the induced action of $\Gamma$ on $\Pp(H)$. Then, there is a fixed point $\bar T$ associated to a smallest unitariser $S$ of $\pi$, such that
\begin{align*}
 d(\bar T,\id_H)=d\left(\Pp(H)^\Gamma,\id_H\right)=\ln(s(S))
\end{align*}
\end{lem}
\begin{proof}\frei\\
 By multiplying the fixed point $T:=SS^*$ corresponding to a smallest positive unitariser $S$ with $\left(\min(\sigma(T))\cdot\max(\sigma(T))\right)^{-\frac12}=\sqrt{\Vert T\Vert^{-1}\cdot \Vert T^{-1}\Vert} 
$
 one gets a fixed point $\bar T$ such that 
\begin{align*}\left(\min\sigma(\bar T)\right)^{-1}&=\left(\left(\min\sigma(T)\cdot\max\sigma(T)\right)^{-\frac12}
\min\sigma(T)\right)^{-1}=\left(\frac{\max\sigma(T)}{\min\sigma(T)}\right)^{\frac12}\\
&=\left(\left(\min\sigma(T)\cdot\max\sigma(T)\right)^{-\frac12}\max\sigma(T)\right)=\max\sigma(\bar T)
\end{align*}
And therefore $\Vert \bar T\Vert=\left\Vert\bar T^{-1}\right\Vert=\sqrt{\Vert T\Vert\cdot\Vert T^{-1}\Vert}$
which in turn implies 
\begin{align*}
s(S)&=\sqrt{s(T)}=\sqrt{s(\bar T)}=\exp(\ln\sqrt{\Vert\bar T\Vert^2})=\exp(d(\id_H,\bar T))
\end{align*}
Besides, operators with such spectral symmetry are precisely those, that realize
\begin{align*}
 \min\left\{\left.\Vert\ln(\alpha S)\Vert\right|\alpha\in\R_+\right\}&=\min\limits_{\alpha\in\R_+}\left\{\max\left\{\ln\Vert\alpha S\Vert,\ln\left(\Vert(\alpha S)^{-1}\Vert^{-1}\right)\right\}\right\}\\
&=\min\limits_{\alpha\in\R_+}\left\{\max\left\{\vert\ln\min(\sigma(\alpha S))\vert,\vert\ln\max(\sigma(\alpha S))\vert\right\}\right\}\\
&=\min\limits_{\alpha\in\R_+}\left\{\max\left\{\begin{array}{l}\vert\ln\alpha+\ln\min(\sigma(S))\vert\\\vert\ln\alpha+\ln\max(\sigma(S))\vert\end{array}\right\}\right\}
\end{align*}
Hence, we can argue conversely that a fixed point $T$ of the $\Gamma$-action $\rho_\pi$ minimizing the distance $d(T,\id_H)$, will have $\Vert T\Vert=\Vert T^{-1}\Vert$ and therefore $d(\id_H,T)=\Vert\ln T\Vert=\ln\Vert T\Vert=\ln\left(s(T)^{\frac12}\right)=\ln s(S)$ (recall that $T=SS^*$ for a smallest unitariser $S$).\par
Thus, we have seen that smallest unitarisers with $\Vert S\Vert=\Vert S^{-1}\Vert$ stand in 1:1-corres-pondence with points in $\Pp(H)^\Gamma$ having minimal distance to $\id_H$ and (by the $\Gamma$-invariance of $d$) to the $\Gamma$-orbit of $\id_H$.
\end{proof}\pagebreak
We can now give an equivalent, geometric version of Theorem \ref{eins}:
\begin{cor}\label{einsgeom}
 Let $\Gamma$ be a unitarisable group. Then, there are universal constants $C(\Gamma)$ and $\alpha(\Gamma)$ depending only on $\Gamma$ such that for any action $\rho_\pi$ of $\Gamma$ on $\Pp(H)$ induced by a uniformly bounded representation $\pi$ on $H$,
\begin{align*}
 d\left(\id_H,\Pp(H)^\Gamma\right)=C(\Gamma)+\frac{\alpha(\Gamma)}2\diam(\pi)
\end{align*}
\end{cor}
\begin{proof}\frei\\
 First of all, by Lemmas \ref{fixedunit} and \ref{infgleichmin}, $\Pp(H)^\Gamma$ is non-empty and the distance $ d\left(\id_H,\Pp(H)^\Gamma\right)$ is realized by some particular $T$ such that (Lemma \ref{symspec}) $d(\id_H,T)=\ln s(S)$ for the smallest unitariser $S=\sqrt{T}$ corresponding to $T$.\par
Now, by Theorem \ref{eins}, there are $K(\Gamma)$ and $\alpha(\Gamma)$ such that $ s(S)\leq K(\Gamma)\vert\pi\vert^{\alpha(\Gamma)}$.\\
Therefore, taking together both results and using Lemma \ref{diamvert}
\begin{align*}
 d\left(\rho_\pi(\Gamma,\id_H),\Pp(H)^\Gamma\right)&=\ln s(S)\leq\ln\left(K\vert\pi\vert^\alpha\right)=\ln K+\alpha\ln\vert\pi\vert=\ln K +\frac\alpha2\diam(\pi)
\end{align*}
Which proves the claim for $C(\Gamma)=\ln K(\Gamma)$.
\end{proof}
The following is due to G. Pisier:
\begin{thm}[\cite{pis2}]\label{zwei}
The following are equivalent for a discrete group $\Gamma$\begin{itemize}
\item $\Gamma$ is amenable. 
\item Theorem \ref{eins} holds for $K(\Gamma)=0,~\alpha(\Gamma)=2$.
\end{itemize}
\end{thm}
This theorem now has a neat geometric translation.
\begin{defn} For a representation $\pi$ of some group $\Gamma$ on the Hilbert space $H$, define
\begin{align*}
 X_\pi&:=\left\{x\in\Pp(H): d(x,\rho_\pi(\gamma ,\id_H))\leq\diam\rho_\pi(\Gamma,\id_H)~\forall \gamma \in \Gamma\right\}\\
&=\left\{x\in\Pp(H): d(x,y)\leq\diam\rho_\pi(\Gamma,\id_H)~\forall y\in\conv\rho_\pi(\Gamma,\id_H)\right\}
\end{align*}\end{defn}
\begin{cor}\label{amenchar}
 A group $\Gamma$ is amenable, if and only if $ X_\pi\cap\Pp(H)^\Gamma\neq\emptyset$ for every uniformly bounded representation $\pi$.
\end{cor}
\begin{proof}
This is an obvious consequence of Theorem \ref{zwei} and Corollary \ref{einsgeom}.
\end{proof}
And Dixmier's question now translates into
\begin{qu}
Is it true, that $ X_\pi\cap\Pp(H)^\Gamma\neq\emptyset$ for every unitarisable $\Gamma$ and every uniformly bounded $\pi$?
\end{qu}
\section{Topological facts about the cone of positive operators}
On $\Pp(H)$ there are two structures: the metric and the linear structure. We will now look at their interplay.
\begin{nota}
 We shall denote by $\tau_d$ the metric topology, by $\tau_{\Vert\cdot\Vert}$ the ordinary norm-topology and the weak operator topology will be denoted by $\tau_w$.\par
Furthermore, we will denote by $\overline A$ the closure of $A$ with respect to the ambient topology. If it is needed, the topology with respect to which we mean $\overline A$ to be closed, will be noted $\overline A^\tau$.
\end{nota}
\subsection{Compactness}
\begin{rem}
 We remark first of all, that all topologies discussed in this chapter are invariant under the $d$-isometries $A\mapsto B^{\frac12}AB^{\frac12}$ with positive and invertible operators $B$.\par
As the space of positive invertible operators is not closed (with respect to any of the topologies discussed here apart from $\tau_d$), one has to keep in mind, that generally speaking $\tau$-convergent nets do not have to have their limit in $\PP(H)$.
\end{rem}
\begin{lem}\label{nextproof}
Open (closed) $d$-balls of radius $r$ around $\id_H$ correspond to open (closed) norm-balls (intersected with the space of positive operators).\par
Therefore, closed $d$-balls of finite radius around $\id_H$ are compact with respect to $\tau_w$.
\end{lem}
\begin{proof}\frei\\
 One sees that
\begin{align}
 B^{d}(\id_H,r)&=\left\{S\in\Pp(H)\left|\Vert\ln S\Vert<r\right.\right\}\nonumber\\
 &=\left\{S\in\Pp(H)\left|\max\left\{\vert\ln\min(\sigma(S))\vert,\vert\ln\max(\sigma(S))\vert\right\}<r\right.\right\}\nonumber\\
 &=\left\{S\in\Pp(H)\left|\sigma(S)\subseteq\left(\exp(-r),\exp(r)\right)\right.\right\}\label{specdefn}
\end{align} 
which gives a spectral definition of $d$-balls around $\id_H\in\Pp(H)$.\par
Furthermore, this yields
\begin{align}
B^{d}(\id_H,r)&=\frac{\exp(-r)+\exp(r)}2\id_H+\nonumber\\&\hspace{0.13cm}+\left\{S=S^*\left|\sigma(S)\subseteq\left(-\frac{\exp(r)-\exp(-r)}2,\frac{\exp(r)-\exp(-r)}2\right)\right.\right\}\nonumber\\
&=\frac{\exp(-r)+\exp(r)}2\id_H+B^{\Vert\cdot\Vert}\left(0,\frac{\exp(r)-\exp(-r)}2\right)\cap\Pp(H)\nonumber\\
&=B^{\Vert\cdot\Vert}\left(\frac{\exp(-r)+\exp(r)}2\id_H,\frac{\exp(r)-\exp(-r)}2\right)\cap\Pp(H)\label{specdefn2}
\end{align}
The same is obviously true, if $<$ is replaced by $\leq$ and open intervals by closed intervals in the calculation above.\par
To prove compactness of closed $d$-balls of radius $r$, one has to see that operators in $B:=\overline{B^{\Vert\cdot\Vert}\left(\frac{\exp(-r)+\exp(r)}2\id_H,\frac{\exp(r)-\exp(-r)}2\right)}$ have spectrum away from $0$ and are therefore invertible.\par
This implies, that the intersection of $B$ with $\Pp(H)$ is the same as its intersection with the $\tau_w$-closed space of positive operators. Hence, as an intersection of a $\tau_w$-compact set with a $\tau_w$-closed set, $\overline{B^{d}(id_H,r)}$ is itself $\tau_w$-compact.
\end{proof}
\begin{thm}\label{sametop}
 The topologies $\tau_d$ and $\tau_{\Vert\cdot\Vert}$ agree on $\Pp(H)$.
\end{thm}
\begin{proof}\frei\\
 In Lemma \ref{nextproof}, we have seen that $d$-balls around $\id_H\in\Pp(H)$ are also balls (of different radius and around different midpoints) with respect to the norm.\par
Conversely, given a radius $\alpha\in(0,1)$ the norm-ball $B^{\Vert\cdot\Vert}(\id_H,\alpha)$ of radius $\alpha$ around $\id_H$ (intersected with $\Pp(H)$) consists of all positive operators with spectrum in the interval $(1-\alpha,1+\alpha)$. Now choose some $r>0$ with $\exp(r)<1+\alpha$, then
\begin{align*}
 1>1-\alpha^2=(1-\alpha)(1+\alpha)&\Rightarrow~1-\alpha<\frac1{1+\alpha}<\exp(-r)\\
&\Rightarrow \left(\exp(-r),\exp(r)\right)\subset(1-\alpha,1+\alpha)
\end{align*}
By (\ref{specdefn}), it is obvious, that $ B^d(\id_H,r)\subset B^{\Vert\cdot\Vert}(\id_H,\alpha)$.\par
We have shown that the local bases at $\id_H$ for the topologies $\tau_d$ and $\tau_{\Vert\cdot\Vert}$ are equivalent in the way that every element of one of the local bases contains a neighbourhood of $\id_H$ from the other topology and both topologies share the transitive subgroup of their homeomorphisms, namely $\{x\mapsto axa^*,a\in B(H)~\textrm{invertible}\}$. Hence both topologies are the same.
\end{proof}
\begin{rem}
 The fact, that the metric topologies $\tau_{\Vert\cdot\Vert}$ and $\tau_d$ coincide on $\Pp(H)$ does not imply, that the corresponding metrics are equivalent in the following sense
\begin{align*}
 \exists c,C:~\forall x,y\in\Pp(H): c\leq\Vert x-y\Vert\leq d(x,y)\leq C\Vert x-y\Vert
\end{align*}
For example, the sequence $(x_n)_{n\in\N},~x_n:=\frac 1n\id_H$ is bounded in norm but not in $\tau_d$.\par
This fact does not contradict the equality of $\tau_d$ and $\tau_{\Vert\cdot\Vert}$, since $\tau_d$ does not come from a Banach topology on $B(H)$ or the space of self-adjoint operators, for which equality of topologies implies equivalence of the corresponding norms (and hence metrics). \par 
But what we do have, is the following
\begin{lem} $d$-bounded subsets of $\Pp(H)$ are norm-bounded.
\end{lem}
\begin{proof}\frei\\
 This is an easy consequence of Lemma \ref{nextproof}:\par If $A\subset\Pp(H)$ is bounded, then so is $A\cup\{\id_H\}$ and hence for some $r>0$
\begin{align*}
 A\subset\overline{B^d(\id_H,r)}\stackrel{(\ref{specdefn2})}\subset\overline{B^{\Vert\cdot\Vert}\left(\frac{\exp(-r)+\exp(r)}2\id_H,\frac{\exp(r)-\exp(-r)}2\right)}
\end{align*}
\end{proof}
\end{rem}
\begin{cor}\label{compactball}
 Closed $d$-balls are $\tau_w$ compact. Hence, every $d$-bounded subset of $\PP(H)$ is $\tau_w$-precompact.
\end{cor}
\begin{proof}\frei\\
Let $B=\overline{B^d(A,r)}$ be a closed $d$-ball. Then, $ B=A^{\frac12}\overline{B^d(\id_H,r)}A^{\frac12}$ and $B$ is the image of the $\tau_w$-compact set $B^d(\id_H,r)$ under a $\tau_w$-continuous map.\par 
Now, if a set $U$ is $d$-bounded, it is contained in a $\tau_w$-compact closed $d$-ball, which is weak operator closed. Hence, it contains the $\tau_w$-closure $\overline U$, which, as a closed subset of a $\tau_w$-compact set is itself $\tau_w$-compact.
\end{proof}
\begin{rem}It is not clear, whether or not $d$-closed, $d$-convex and $d$-bounded sets are $\tau_w$ closed. What we do know, though, is the following \end{rem}
\begin{prop}\label{heineborel}
 Let $(x_n)\subset\Pp(H)$ be a $\tau_w$-convergent sequence such that for some $y\in\Pp(H)$ and $N\in\N$ we have $d(x_n,y)<\alpha,~\forall n>N$. Then, this is also true for the limit point $x_0$ of $(x_n)$.\par
 In particular, if $A\subset\Pp(H)$ is $d$-closed, $d$-convex and $d$-bounded, then any $\tau_w$-limitpoint $x$ of some sequence $(x_n)_{n\in\N}\subset A$ fullfills $ d(x,y)\leq\diam(A)~\forall y\in A$
\end{prop}
\begin{proof}\frei\\
 The sequence $(x_n)$ lies in the $\overline{B^d(y,\alpha)}$, which is $\tau_w$-compact by Lemma \ref{compactball}.
\end{proof}
So, generally speaking, $\tau_w$-limit points of sequences inside $d$-convex and $d$-closed sets are ``not too far away'' from the sequence. This motivates the following definition:
\begin{defn}
 We say that a point $x$ is \textbf{convex close} to a subset $A$ of a metric space $X$, if $d(x,a)\leq\diam(A)~\forall a\in A$.
\end{defn}
In the last section, we introduced the space $X_\pi$ as the set of points convex close to $\rho_\pi(\Gamma,\id_H)$.\par
With the help of Proposition \ref{heineborel}, we may now collect some facts about this space:
\begin{prop}\label{xpi}
 Let $\pi:\Gamma\to B(H)$ be a uniformly bounded representation. Then the space $X_\pi\subset\Pp(H)$ is a $\tau_w$-compact $\Gamma$-space, which is $d$-convex, $d$-bounded and $d$-closed.
\end{prop}
\begin{proof}\frei
$\tau_w$-closedness follows from Proposition \ref{heineborel}. Since $X_\pi$ is also bounded, it is $\tau_w$- compact.\par
Now, let $x\in X_\pi$. Then for any $\gamma\in \Gamma$ one has (due to the invariance of $d$ under $\rho$) $ d(\gamma x,\rho(\gamma',\id_H))=d(x,\rho(\gamma ^{-1}\gamma',\id_H))\leq\diam\rho_\pi(\Gamma,\id_H)~\forall \gamma'\in \Gamma$ and therefore $\gamma x\in X_\pi$ proving that $X_\pi$ is a $\Gamma$-space.\par
$d$-boundedness and $d$-closedness are obvious and
\begin{align*}
 d(\eta(x,y,t),\rho_\pi(\gamma ,\id_H))&\leq (1-t)d(x,\rho_\pi(\gamma ,\id_H))+td(x,\rho_\pi(\gamma ,\id_H))\\
&\leq (1-t)\diam\rho_\pi(\Gamma,\id_H) +t\diam\rho_\pi(\Gamma,\id_H)\\
&=\diam\rho_\pi(\Gamma,\id_H)
\end{align*} implies the metric convexity of $X_\pi$.
\end{proof}
\subsection{Midpoints and circumradii}\parbox{1cm}{}\\
For a bounded subset $A$ of a Banach space, there exists a unique $r$, such that $A$ is contained in a ball of radius $r$. We show in this section, that this is also a property of $\Pp(H)$ with its metric topology coming from $d$.
\begin{defn}
 Let $U\subset \PP(H)$ be a $d$-bounded and $d$-convex set. \\We define the \textbf{circumradius} of $U$ to be $\inf\limits_{r\in\R}\left\{\exists x_r\in\Pp(H):~U\subset\overline{B(x_r,r)}\right\}$.\par
If for the circumradius $r^*$ of $U$ there is some $x^*$ such that $U\subset\overline{B(x^*,r^*)}$, we call $x^*$ a \textbf{midpoint} of $U$.
\end{defn}
\begin{rem}
 Let $U$ be bounded with diameter $l$. Then obviously $\frac l2\leq r\leq l$ for the circumradius $r$. Also, the midpoints of $U$ are obviously convex close to $U$.
\end{rem}
\begin{lem}
 Midpoints and circumradii exist for every bounded set $U$.
\end{lem}
\begin{proof}\frei\\
 Let $r$ be the circumradius of $U$. For $n\in\N$ define $r_n$ by $r_n=r+\frac1n$ and let $(x_n)_{n\in\N}$ be a corresponding sequence of points $x_n\in\Pp(H)$, such that $B(x_n,r_n)\supset U$.\par
Then, by applying Proposition \ref{heineborel}, we see that $\tau_w$-limit points of this sequence are midpoints.
\end{proof}
 A priori, the set of midpoints does not have to be a singleton. But the following holds:
\begin{lem}\label{midp}
 For any bounded set $U$ with circumradius $r$, the set $M(U)$ of midpoints is $d$-convex, $d$-closed and bounded.
\end{lem}
\begin{proof}\frei\\
 Let $x_1$ and $x_2$ be two midpoints of $U$, then by the convexity of $d$ for any $t\in[0,1]$
\begin{align*}
 d\left(\eta(x_1,x_2,t),u\right)\leq td(x_2,u)+(1-t)d(x_1,u)\leq t\cdot r+(1-t)r=r~\forall u\in U
\end{align*}
hence $\eta(x_1,x_2,t)\in M(U)$, which shows $d$-convexity of $M(U)$.\par 
The boundedness of $M(U)$ is obvious as for any $y\in U$ and $x_1,x_2\in M(U)$ we have
$d(x_1,x_2)\leq d(x_1,y)+d(x_2,y)\leq 2r$.\par
Now, let $(x_n)_{n\in\N}$ be a sequence in $M(U)$ converging to $x$ with respect to $d$. Then
we get $d(x,u)\leq d(x,x_n)+d(x_n,u)\leq d(x,x_n)+r\limto\limits_{n\to\infty}r$, which shows that $x\in M(U)$. Hence $M(U)$ is a $d$-convex, $d$-bounded and $d$-closed set.
\end{proof}
\begin{rem} One cannot assume, that there is only one midpoint for arbitrary bounded sets as the following example shows:
\end{rem}
\begin{examp}\parbox{1cm}{}\\
 Let $\Gamma$ be a non-unitarisable group and $\pi:\Gamma\to B(H)$ be a uniformly bounded, non-unitarisable representation. Let us consider the orbit $X:=\rho_\pi(\Gamma,\id_H)$ of the identity with respect to the action of $\Gamma$ on $\PP(H)$ induced by this representation.\par 
Since $X$ is bounded (by the uniform boundedness of $\pi$), it has a circumradius, which we will denote by $r$.\par 
Now, from
\begin{align*}
x\in M(X)&~\Rightarrow~~~d(x,\gamma_1 \gamma_1 ^*)\leq r~\forall \gamma_1 \in \Gamma~~~\Leftrightarrow~~~ d\left(x,\gamma_2^{-1}\gamma_1 \gamma_1 ^*\left(\gamma_2^*\right)^{-1}\right)\leq r\forall \gamma_1 ,\gamma_2\in \Gamma\\
&~\Leftrightarrow ~~~d(\gamma_2x\gamma_2^*,\gamma_1\gamma_1^*)\leq r\forall \gamma_1,\gamma_2\in \Gamma~~~\Rightarrow ~~~\gamma_2x\gamma_2^*\in M(X)\forall \gamma_2\in \Gamma
\end{align*}
we see, that $\rho_\pi$ restricts to an action on $M(X)$. So, if there was only one midpoint, it would be fixed by the action of $\Gamma$ and hence by Lemma \ref{fixedunit}, this would imply unitarisability of $\pi$.
\end{examp}
The following example shows, that even in the linear case, the midpoints discussed above are counter-intuitive:
\begin{examp}\parbox{1cm}{}\\
 Consider the set $A=\{0,\delta_n|n\in\N\}\subset\ell^\infty(\N)$, where $\delta_n$ characteristic function of $n\in\N$.\par 
Now, the (algebraic) convex hull $A$ consists of all finitely supported functions with values in $[0,1]$ such that the $\ell^1$-norm is $1$.\par 
Closing this in the $\ell^\infty$-norm means adding those functions of $\ell^1$-norm $1$ taking values in $[0,1]$ and vanishing at infinity.\par 
This set $\bar A$ is obviously convex, $\ell^\infty$-closed and has ``inner'' circumradius $1$: 
\begin{align*}
f\in \bar A\Rightarrow \lim\limits_{x\to\infty}f(x)=0\Rightarrow \Vert f-f_n\Vert_\infty\limto\limits_{n\to\infty}1\end{align*}
so that every point in $\bar A$ has an ``opposite'' point within $U$. In other words: midpoints in $\bar A$ would imply the circumradius to be $1$. \par 
The circumradius ``from the outside'' is less: let $g $ be the constant function with value $\frac12$. Then for every $f\in \bar A,~\frac12\geq\left\vert f(x)-g(x)\right\vert$, hence $\Vert f-g\Vert_\infty=\frac12$.\par
In other words, the ``true midpoints'' (those realizing the smallest possible radius of a ball containing $\bar A$) do not have to be inside $\bar A$, even if $\bar A$ is convex!
\end{examp}
\begin{rem}
 In the sequel, compact will always refer to $\tau_w$-compact and convexity and boundedness are meant be $d$-convexity and $d$-boundedness respectively.
\end{rem}
\begin{lem}\label{sab}
For a compact set $A\subset\Pp(H)$ and a closed set $B\subset\Pp(H)$, there exist points $a\in A$, realizing the distance to $B$: $d(a,B)=d(A,B)$.\par
If, moreover, $B$ is compact, there are points $a\in A$ and $b\in B$, which realize the distance between $A$ and $B$: $d(a,b)=d(A,B)$.
\end{lem}
\begin{proof}\frei\\
Again, this is an easy consequence of Proposition \ref{heineborel}, which assures, that weak limit points are at most as far away from some point, as the limit of the distances prescribes.
\end{proof}
\begin{nota}
 For a compact subset $A$ of $\Pp(H)$ and a closed subset $B\subset\Pp(H)$, we denote by $S(A,B)\subset A$, the set of points $a\in A$, for which
\begin{align*}
 d(a,B)=d(A,B)
\end{align*}
Those points exist by Lemma \ref{sab}.
\end{nota}
As a final result of this section, we observe, that the sets constructed above are "nice" $\Gamma$-subspaces of $\Pp(H)$:
\begin{prop}
 Let $A$ be a bounded $\Gamma$-subset of $\PP(H)$. Then the sets $M(A)$ and $S(A,B)$  are bounded, closed and convex $\Gamma$-subsets of $\PP(H)$.
\end{prop}
\begin{proof}\frei\\
 Convexity does always follow from the convexity of $x\mapsto d(x,y)$, which implies that for two points $a$ and $b$ equally far away from a third point $c$, elements $\eta(a,b,t),~t\in[0,1]$ are at most as far away from $c$ as $a$ and $b$.\par
 In Lemma \ref{midp}, we have seen, that $M(A)$ is a $\Gamma$-space. Now, for a point $a$ in $S(A,B)$, we have $d(a,B)=d(A,B)$ and therefore
\begin{align*}
d(\gamma ^*a\gamma ,B)=\inf\limits_{b\in B}d(\gamma ^*a\gamma ,b)=\inf\limits_{b\in B}d(a,{\gamma ^{-1}}^*b\gamma ^{-1})=d(a,B)=d(A,B)
\end{align*}
showing, that $S(A,B)$ is a $\Gamma$-space, which is obviously bounded as a subset of $A$.\par 
Moreover, if $(x_n)_{n\in\N}$ is a $d$-convergent sequence in $S(A,B)$, the limit $x$ lies in $A$ ($A$ is closed!) and
\begin{align*}
 d(x,B)=\inf\limits_{b\in B}d(x,b)=\inf\limits_{b\in B}\lim\limits_{x\to\infty}d(x_n,b)=\inf\limits_{b\in B}\lim\limits_{x\to\infty}d(A,B)=d(A,B)
\end{align*}
Hence $x\in S(A,B)$ which implies that $S(A,B)$ is closed.
\end{proof}
\section{GCB-spaces and barycenters}
In the sequel, we will generalize the metric structure on $\Pp(H)$ to the concept a GCB-space. Also complete CAT(0)-spaces (broadly discussed by Martin Bridson in \cite{brid}) are GCB-spaces, which in turn are special cases of ``continuous midpoint spaces'' as discussed in \cite{fixed}. In those spaces, we will construct barycenters for finite sets and from this derive a fixed-point theorem for amenable groups.
\subsection{GCB-spaces}
\begin{defn}
 On a metric space $X$, such that there exist geodesics between any two points, a \textbf{geodesic bicombing} is a map $\eta:X\times X\times [0,1]\to  X$, such that
 \begin{itemize}
\item the map $\eta(x,y,\cdot):[0,1]\to X$ is a geodesic for any $(x,y)\in X\times X$
\item $\eta(y,x,t)=\eta(x,y,1-t)~\forall t\in I,~\forall x,y\in X$
\item $\eta(x,\eta(x,y,t),s)=\eta(x,y,ts)~\forall s,t\in I,~\forall x,y\in X$
\item $\lim\limits_{n\to\infty}\eta(x_n,y_n,t)=\eta\left(\lim\limits_{n\to\infty}x_n\lim\limits_{n\to\infty},y_n,t\right)$ for all $t\in I$ and convergent sequences $(x_n)$ and $(y_n)$.
\end{itemize}
Let $(X,d)$ and $(Y,d')$ be two spaces with a distinguished geodesic bicombing. Then, a map $f:(X,d)\to(Y,d')$ is said to be \textbf{bicombing respecting}, if
\begin{align*}f\circ\eta(x,y,t)=\eta(f(x),g(y),t)~\forall x,y\in X,~\forall t\in[0,1]\end{align*}
\end{defn}
\begin{defn}
 A \textbf{GCB-space} is a complete metric space $(X,d)$ together with a fixed geodesic bicombing $\eta$, such that the metric is convex with respect to this bicombing (i.e., equation (\ref{gammaconv}) holds).
\end{defn}
\begin{nota}
A any set $A\subset X$, $\conv(A)$ denotes the smallest closed and convex set containing $A$.
\end{nota}
On an arbitrary GCB-space, there is no such thing as a ``natural'' weak toplogy $\tau_w$, which has shown to be very fruitful in the case of $\Pp(H)$.\par 
 The following property will make up for this at some points
\begin{defn}
 We say, a GCB-space $X$ has \textbf{property (C)}, iff the following holds\par 
\fbox{\parbox{14cm}{Given a bounded sequence $(x_n)_{n\in\N}$ in $X$ and a family $\{f_\alpha|\alpha\in I\}$ of isometries of $X$ respecting the geodesic bicombing ($I$ is any index set) such that $d(x_n,f_\alpha(x_n))\to0$ for any $\alpha\in I$, there is some $x\in X$ convex close to the sequence $(x_n)_{n\in\N}$ such that \begin{center}$x=f_\alpha(x)~\forall\alpha\in I$\end{center}
}}\end{defn}
 The point $x$ in property (C) is not necessarily a $d$-limit point (for which the latter property is obvious):
\begin{lem}\label{ppisgbc}
 For a Hilbert space $H$ and with the definitions from above, $\PP(H)$ is a GCB-space with property (C) when considering only isometries $f_A:x\mapsto A^*xA$ for $A\in B(H)$.
\end{lem}
\begin{proof}\frei\\
 We only need to show property (C).\par 
Given a bounded sequence $(x_n)_{n\in\N}$, $X:=\conv(\{x_n|n\in\N\})$ is convex, bounded and closed. By Proposition \ref{heineborel}, we find a $\tau_w$-limit point $x$ convex close to $X$. In generally, this point does not have to be a $d$-limit point.\par 
Now given a family $\mathcal F$ of isometries $f_A$ and the assumption in property (C) by definition of $d$, $d(x_n,f_A(x_n))\limto\limits_{n\to\infty}0 $ holds for any $f_A\in\mathcal F$ and by Theorem \ref{sametop}, this implies the convergence in norm: $\|A^*x_nA-x_n\|\limto\limits_{n\to\infty}0$.\par
Hence, we have for any $x,y\in H$
\begin{align*}
\left|\langle(A^*xA-x)u,v\rangle\right|&=\lim\limits_{n\to\infty}\left|\langle (A^*x_nA-x_n)u,v\rangle\right|\leq\lim\limits_{n\to\infty}\left\|A^*x_nA-x_n\right\|\|u\|\|v\|=0\end{align*}
by the fact, that $\lim\limits_{n\to\infty}x_n=x$ with respect to $\tau_w$.\par 
This was true for any $u,v\in H$ so that $f_A(x)=x$ for arbitrary $f_A\in\mathcal F$.
\end{proof}
\begin{examp}\parbox{1cm}{}\\
 For a reflexive Banachspace $(X,\|\cdot\|)$, a geodesic bicombing can be defined by
\begin{align*}\eta(x,y,t)=tx+(1-t)y,~t\in I,~x,y\in X\end{align*}
The triangle inequality yields convexity of this bicombing and weak limit points comply with property (C). Hence, $X$ is a GCB-space with property (C).
\end{examp}
Complete CAT(0) spaces are called Hadamard spaces, they form another class of GCB-spaces (they are easily seen to be uniquely geodesic. Hence they carry a natural geodesic bicombing). Compact, closed subspaces will also have property (C).
\begin{rem}
 Obviously, points in $\conv(A)$ are convex close to $A$ and in all the examples above apart from $\Pp(H)$, we may find the point $x$ from property (C) to lie inside the closed convex hull $\conv\{x_n,n\in\N\}$.
\end{rem}
\subsection{Barycenters of finite sets}
\begin{nota}
 We will denote by $[n]$ the set $\{1,..,n\}\subset\N$.
\end{nota}\begin{defn}
 We define the \textbf{$n$-tuple space} of a topological space $X$ to be
\begin{align*}X_{(n)}=\left.\prod\limits_{i\in[n]}X\right/S_n\end{align*}
the space of unordered $n$-tuples. ($S_n$ denotes the symmetric group on $n$ elements)\par 
Elements in the $n$-tuple-space are denoted by $(x_1,..,x_n)$ or by $(x_i,i\in[n])$.
\end{defn}
\begin{rem}
By defining 
\begin{align*}
 d_{(n)}:X_{(n)}\times X_{(n)}\to\R,~\left((x_i,i\in[n]),(y_i,i\in[n])\right)\mapsto\min\limits_{\sigma\in S_n}\frac1n\sum\limits_{i\in[n]}d\left(x_i,y_{\sigma(i)}\right)
\end{align*}
$X_{(n)}$ is turned into a complete metric space.
\end{rem}
\begin{rem}
 To any $\phi:X\to X$, we set $\tilde\phi:X_{(n)}\to X_{(n)},~\left(x_i,i\in[n]\right)\mapsto\left(\phi(x_i),i\in[n]\right)$.
\end{rem}
\begin{defn}
A map $b_n:X_{(n)}\to X$ is called a \textbf{barycenter map}, if
\begin{enumerate}
\item  $b_n((x_1,..,x_n))\in\conv(\{x_1,..,x_n\})$
\item $b$ is equivariant with respect to bicombing-respecting maps $\phi:X\to X$,	
\begin{center}
 i.e.\hspace{0.5cm}$\xymatrix{
X_{(n)}\ar[r]^{\tilde\phi}&X_{(n)}\ar[d]^b\\
X\ar@{<-}[u]^b&X\ar@{<-}[l]^\phi
}$\hspace{0.5cm}commutes
\end{center}
\end{enumerate}
\end{defn}
\begin{defn}
 The image of an $n$-tuple by a barycenter map is called a \textbf{barycenter} of this tuple.
\end{defn}
\begin{rem}
  Even though the barycenter map is a map of tuples, we will frequently speak of ``barycenters of a subset of $X$''. A set $\{x_1,..,x_n\}$ is then identified with the obvious corresponding tuple $(x_i,i\in[n])$.
 Vice versa, one associates to an $n$-tuple over $X$ the subset containing all points from the tuple.\par
 Therefore, it is possible, to associate to an $n$-tuple $A$ the closed convex hull $\conv(A)\subset X$ or the diameter $\diam(A)$ of $A$.\par 
In particular, $x\in(x_1,..,x_n)$ says that there is some $i\in[n]$ such that $x=x_i$.
\end{rem}\begin{rem}
 For $n\in\{1,2\}$, there are obvious choices for barycenter maps:
\begin{align*}b_1:&X_{(1)}\to X,(x)\mapsto x\hspace{1cm}\textrm{and}\hspace{1cm}b_2:~X_{(2)}\to X, (x,y)\mapsto \eta\left(x,y,\frac12\right)\end{align*}
\end{rem}
In fact, there was no choice: $x$ is the only point in the closed convex hull of $x$ and exchanging $x_1$ and $x_2$ in the definition of $b_2$ had to leave the result invariant.
\begin{defn}
A map $f:(X,d)\to(Y,d')$ is \textbf{non-expansive}, if $d'(f(x),f(y))\leq d(x,y)$ for any $x,y\in X$. Those maps are obviously continuous.\end{defn}
\begin{thm}\label{2finish}
For any GCB-space $X$ and any $n\in\N$ there exists a non-expansive barycenter map $b_n:X_{(n)}\to X$.
\end{thm}
\begin{proof}\frei\\
We proceed by induction assuming that we have already defined a non-expansive $n$-barycenter map $b_n$. (The initial step $n=2$ is obvious and follows from the convexity of the geodesic bicombing)\par 
Let us define the following auxiliary map:
\begin{align*}
  \tilde b_{n+1}:&\hspace{0.8cm}X_{(n+1)}\hspace{0.7cm}\to X_{(n+1)}\\
&(x_i,i\in[n+1])\mapsto \left(b_n\left(x_j,i\neq j\in[n+1]\right),i\in[n+1]\right)
\end{align*}
This map is obviously well-defined and equivariant with respect to bicombing-respecting maps.
The proof will follow from the following lemma:
 \begin{lem}\label{technical}
 $\tilde b_{n+1}$ has the following properties:
\begin{enumerate}
  \item $\diam\left(\tilde b_{n+1}(A)\right)\leq\frac 1{n}\diam A~\forall A\in X_{(n+1)}$
\item $\tilde b_{n+1}:X_{(n+1)}\to X_{(n+1)}$ is non-expansive.
\end{enumerate}\end{lem}
Let us first see, how the theorem follows:\par 
Using the first property, one directly sees that $\left(\conv\left((b_{n+1})^k(A)\right)\right)_{k\in\N}$ is a nested sequence of convex and closed sets with diameter $\frac1{n^k}\diam(A)$. and using the completeness of $X_{(n)}$, one immediately gets that the limit map
\begin{align*}
 \hat b_{n+1}:X_{(n+1)}\to X_{(n+1)},~(x_i,i\in[n+1])\mapsto\lim\limits_{k\to\infty}(\tilde b_{n+1})^k\left(x_i,i\in[n+1]\right)
\end{align*}
is well-defined, and maps every $(n+1)$-tuple to a tuple of diameter 0. \par 
Also, as a limit of equivariant and non-expansive maps, it is itself equivariant and non-expansive (and in particular continuous). Hence we have $\hat b_{n+1}(A)=(x(A),..,x(A))$ for some $x(A)\in\conv(A)$ and define $b_{n+1}(A):=x(A)$. \par 
Now, $b_{n+1}$ is well-defined and equivariant with respect to maps respecting the bicombing (as $\tilde b_{n+1}$ and hence $\hat b_{n+1}$ is) and from
\begin{align*}
&d\left(b_{n+1}((x_i,i\in[n+1])),b_{n+1}((y_i,i\in[n+1]))\right)\\
&\hspace{2cm}=d_{(n+1)}\left(\lim\limits_{k\to\infty}\tilde b_{n+1}^k(x_i,i\in[n+1]),\lim\limits_{k\to\infty}\tilde b_{n+1}^k(y_i,i\in[n+1])\right)\\
&\hspace{2cm}=\lim\limits_{k\to\infty}d_{(n+1)}\left(\tilde b_{n+1}^k((x_i,i\in[n+1])),\tilde b_{n+1}^k((y_i,i\in[n+1]))\right)\\
&\hspace{2cm}\leq\lim\limits_{k\to\infty}d_{(n+1)}\left((x_i,i\in[n+1]),(y_i,i\in[n+1])\right)\\
&\hspace{2cm}=d_{(n+1)}\left((x_i,i\in[n+1]),(y_i,i\in[n+1])\right)
\end{align*}
one sees, that $b_{n+1}$ is non-expansive.\end{proof}
\begin{proof}\textit{(of Lemma \ref{technical})}\\
We show both properties individually:
  \begin{enumerate}
    \item Let $A=(x_i,i\in[n+1])\in X_{(n+1)}$ and $y_1\neq y_2\in\tilde b_{n+1}(A)$ be arbitrary. Then, by definition of $\tilde b$ there are $j\neq k\in[n+1]$ such that $y_1=b_n\left((x_1,..,\hat x_j,..,x_{n+1})\right)$ and $y_2=b_n\left((x_1,..,\hat x_k,..,x_{n+1})\right)$ and one easily sees by using the non-expansiveness of $b_n$ and the definition of $d_{(n)}$, that
\begin{align*}
  d(y_1,y_2)&=d\left(b_n\left((x_1,..,\hat x_j,..,x_{n+1})\right),b_n\left((x_1,..,\hat x_k,..,x_{n+1})\right)\right)\\
&\leq d_{(n)}\left((x_1,..,\hat x_j,..,x_{n+1}),(x_1,..,\hat x_k,..,x_{n+1})\right)\\
&\leq\frac1n\left(d(x_k,x_j)+\sum\limits_{i\in[n+1]\setminus\{j,k\}}d(x_i,x_i)\right)=\frac1n{d(x_k,x_j)} \leq\frac{\diam A}n
\end{align*}
Since this was true for arbitrary $y_1$ and $y_2\in\tilde b_{n+1}(A)$, $\diam\left(\tilde b_{n+1}(A)\right)\leq\frac 1{n}\diam A$.
\item Let $(x_i,i\in[n+1])$ and $(y_i,i\in[n+1])$ be arbitrary $(n+1)$-tuples. Choose the labelling in such a way that
\begin{align*} 
  d_{(n+1)}\left((x_i,i\in[n+1]),(y_i,i\in[n+1])\right)&=\frac1{n+1}\sum\limits_{i\in[n+1]}d(x_i,y_i)
\end{align*}
To abbreviate notation, let $U_k:=[n+1]\setminus\{k\}$ and $\bar x_k=(x_i,i\in U_k)\in X_{(n)}$. Also, let $\Perm(U_k)$ denote the group of all permutations of $U_k$.\par
Then, by using the non-expansiveness of $b_n$ (by induction), we see
\begin{align*}
  &d_{(n+1)}\left(\tilde b_{n+1}((x_i,i\in[n+1])),\tilde b_{n+1}((y_i,i\in[n+1]))\right)\\
&\hspace{2cm}=d_{(n+1)}\left(\left(b_n\left(\bar x_k\right),k\in[n+1]\right),\left(b_n\left(\bar y_k\right),k\in[n+1]\right)\right)\\
&\hspace{2cm}=\min\limits_{\sigma\in S_n}\frac 1{n+1}\sum\limits_{k\in[n+1]}d\left(b_n\left(\bar x_k\right),b_n\left(\bar y_{\sigma(k)}\right)\right)\\
&\hspace{2cm}\leq\frac 1{n+1}\sum\limits_{k\in[n+1]}d\left(b_n\left(\bar x_k\right),b_n\left(\bar y_k\right)\right)\leq\frac 1{n+1}\sum\limits_{k\in[n+1]}d_{(n)}\left(\bar x_k,\bar y_k\right)\\
&\hspace{2cm}=\frac 1{n+1}\sum\limits_{k\in[n+1]}\min\limits_{\tau\in\Perm(U_k)}\frac1n\sum\limits_{j\in U_k}d\left(x_j,y_{\tau(j)}\right)\leq\frac 1{n+1}\sum\limits_{k\in[n+1]}\frac1n\sum\limits_{j\in U_k}d\left(x_j,y_j\right)\\
&\hspace{2cm}=\frac1{n+1}\sum\limits_{k\in[n+1]}d(x_k,y_k)=d_{(n+1)}\left((x_i,i\in[n+1]),(y_i,i\in[n+1])\right)
\end{align*}
  \end{enumerate}
\end{proof}\begin{nota}$A\dcup B$ denotes the disjoint union of $A$ and $B$.\end{nota}
\begin{cor}\label{foln}
Let $A,B,C\subset X$ be finite subsets of a GCB-space with $\vert B\vert=\vert C\vert$. Then, for the corresponding barycenter map $b$: $d\left(b\left(A\dcup B\right),b\left(A\dcup C\right)\right)\leq\frac{|B|}{\vert A\vert+\vert B\vert}\diam\left(B\dcup C\right)$.
\end{cor}
\begin{proof}\frei\\
Choose some bijection $\sigma':B\to C$ and define $\sigma: A\dcup B\to A\dcup C$ to be the identity on $A$ and $\sigma'$ otherwise.\par
Then, since the barycenter map is non-expansive, one sees (here $G$ shall denote the group of all bijections $A\dcup B\to A\dcup C$ and $n=\vert A\vert+\vert B\vert$)
\begin{align*}
d\left(b\left(A\dcup B\right),b\left(A\dcup C\right)\right)&\leq d_n\left(A\dcup B,A\dcup C\right)=\min\limits_{\phi\in G}\frac1n\sum\limits_{x\in A\dcup B}d\left(x,\phi(x)\right)\\
&\leq\frac1n\sum\limits_{x\in A\dcup B}d\left(x,\sigma(x)\right)=\frac1n\left(\sum\limits_{x_i\in A}d\left(x_i,x_i\right)+\sum\limits_{y_i\in B}d\left(y_{i},\sigma'\left(y_i\right)\right)\right)\\
&=\frac1n\sum\limits_{y_i\in B}d\left(y_i,\sigma'\left(y_i\right)\right)\leq\frac{|B|}n\diam\left(B\dcup C\right)
\end{align*}
\end{proof}
One could wonder, whether the barycenter maps defined above respect the GCB-structure in the sense that they send tuples of geodesics to a geodesic.\par 
The following propopsition shows, that this is true for any $n\in\N$, if it holds for $n=2$.
\begin{prop}\label{baryisom}
Let $X$ be a GCB space such that for every $a,b,c,d\in X$ and $t\in I$ we have $b_2(\eta(a,b,t),\eta(c,d,t))=\eta(b_2(a,c),b_2(b,d),t)$.\par
Then, also $b_n\left(\left(\eta(x_i,y_i,t),~i\in[n]\right)\right)=\eta(b_n(A),b_n(B),t)$ holds for any $n\in\N$, where \\$A=(x_i,~i\in[n])$ and $B=(y_i,~i\in[n])\in X_{(n)}$.
\end{prop}
\begin{proof}\frei\\
We prove this by induction the first step $n=2$ being assumed.\par 
Then, by construction, the barycenter $b_{n}\left(\left(\eta(x_i,y_i,t),~i\in[n]\right)\right)$ is the $d$-limit of the sequence $(z_i(t))_{i\in\N}=z_i^{(1)}(t)$, where
\begin{align*}
z_i^{(k)}(t)=\left\{\begin{array}{ll}
\eta(x_k,y_k,t)&i=1\\
b_{n-1}\left(\left(z_{i-1}^{(l)}(t),~k\neq l\in[n]\right)\right)&n\neq 1
\end{array}\right.
\end{align*}
By induction, we know, that $z_i(t)=\eta(z_i(0),z_i(1),t)$.\par 
Using the fact, that $z_i(0)$ and $z_i(1)$ converge to the barycenters $b_n((x_i[i\in[n]))$ and $b_n((y_i,i\in[n]))$ respectively and the continuity of the geodesic bicombing, we see, that
\begin{align*}
 b_n\left(\left(\eta(x_i,y_i,t),~i\in[n]\right)\right)&=\lim\limits_{i\to\infty}z_i(t)
 =\lim\limits_{i\to\infty}\eta(z_i(0),z_i(1),t)
=\eta\left(\lim\limits_{i\to\infty}z_i(0),\lim\limits_{i\to\infty}z_i(1),t\right)\\&=\eta\left(b_n((x_i,i\in[n]]),b_n([y_i,i\in[n])),t\right)
\end{align*}
which proves the claim.
\end{proof}
\subsection{Fixed points and amenable groups}
In this section, we will prove that bicombing-respecting actions by discrete countable groups on GCB spaces with property (C) have fixed points convex close to any bounded orbit, if the action restricted to the orbit is amenable.\par 
Amenable actions by a group $\Gamma$ on a space $X$ are normally defined as actions allowing for $\Gamma$-invariant means (see \cite{glasm} for example). As proven for examble by Rosenblatt in \cite{ros}, this is equivalent to the following definition:
\begin{defn}
 We say, that an action of a countable discrete group $\Gamma$ on a set $X$ is called an \textbf{amenable action}, if for any finite $S\subset \Gamma$ and any $\epsilon>0$, one can find a finite set $A\subset X$, such that $|A\Delta \gamma A|<\epsilon|A|$ for all $\gamma \in S$.\par
A group $\Gamma$ is  an \textbf{amenable group}, if the action of $\Gamma$ on itself by multiplication on the left is amenable.
\end{defn}
\begin{rem}
 Amenable groups always act amenably.
\end{rem}

\begin{defn}
 Let $\Gamma$ act amenably on $X$. Since $\Gamma$ is countable, it is an ascending union of finite sets $U_n$. Let $\epsilon_n=\frac1n$, then the corresponding sequence $\left(F_n\right)_{n\in\N}$ of subsets of $\Gamma$ such that $|F_n\Delta \gamma F_n|<\epsilon_n|F_n|$ for all $\gamma \in U_n$ is called \textbf{F\o lner sequence} for this action.
\end{defn}
\begin{thm}\label{existfixed}
Let $X$ be a GCB-Space with Property (C) and $\Gamma$ be a group acting on $X$ bicombing respectingly, such that the action allows for at least one bounded orbit $\Gamma x$ and restricts to an amenable action on this orbit. Then there is a fixed point $x$ convex close to $\conv(\Gamma x)$.
 \end{thm}
\begin{proof}\frei\\
 Let $F_n\subset \Gamma x$ be a F\o lner-sequence for the restricted action of $\Gamma$ on $\Gamma x$.\par 
Consider the sequence $\left(x_n\right)_{n\in\N}:=\left(b_{\vert F_n\vert}\left(F_n\right)\right)_{n\in\N}$ in $\conv(\Gamma x)$. By construction, any $\gamma $ lies in $U_n$ $\forall n>N_\gamma$ (with $N_\gamma$ big enough) and by the definition of $F_n$ and Corollary \ref{foln} we get for any $\gamma \in \Gamma$
\begin{align*}
d(x_n,\gamma x_n)&=d\left(b_{\vert F_n\vert}\left(F_n\right),\gamma b_{\vert F_n\vert}\left(F_n\right)\right)=d\left(b_{\vert F_n\vert}\left(F_n\right),b_{\vert F_n\vert}\left(\gamma F_n\right)\right)\\
&=d\left(b_{\vert F_n\vert}\left(F_n\cap \gamma F_n\cup\left(F_n\setminus \gamma F_n\right)\right),b_{\vert F_n\vert}\left(F_n\cap \gamma F_n\cup\left(\gamma F_n\setminus F_n\right)\right)\right)\\
&\leq\frac{\left\vert F_n\setminus \gamma F_n\right\vert}{\vert F_n\vert}\diam\left(F_n\Delta \gamma F_n\right)\leq\frac{|F_n\Delta \gamma F_n|}{2|F_n|}\diam (\gamma F_n\cup F_n)\\
&\leq\frac{|F_n\Delta \gamma F_n|}{2|F_n|}\diam(\Gamma x)\limto\limits_{n\to\infty}0
\end{align*}
By definition, Property (C) implies the existence of some $\tilde x$ being convex close to the closed convex hull $\conv(\{x_n,n\in\N\})$ such that $\gamma \tilde x=\tilde x$ for any $\gamma \in \Gamma$ and
\begin{align*}
 d(\tilde x,\conv(\Gamma x))&\leq d(\tilde x,\conv\{x_n,n\in\N\}\leq\diam\conv(\{x_n,n\in\N\})\leq\diam\conv(\Gamma x)
\end{align*}
\end{proof}
\begin{cor}\label{olkj}
 Let $\Gamma$ be an amenable group acting on $\Pp(H)$ by bicombing-respecting isometries. Then, there is a fixed point in $X_\pi$. In particular, this implies that amenable groups are unitarisable as well as one direction in Corollary \ref{amenchar} and Theorem \ref{zwei}.
\end{cor}
\begin{proof}\frei\\
 Apply the above result to $X=X_\pi$.
\end{proof}
One could wonder, if for non-unitarisable groups (or possibly for unitarisable and non-amenable groups, where the fixed point to some group action on $\Pp(H)$ is far away from the $\Gamma$-orbit of $\id_H$), one may find a model for the classifying space (defined in \cite{lueck}, for example) as a bounded subspace of $\Pp(H)$.\par 
The following corollary gives a partial answer to this. The reader may be reminded that an action  of a group on a space $X$ is \textbf{free}, if $\gamma_1x=\gamma_2x$ for some $\gamma_1,\gamma_2\in \Gamma,~x\in X$ implies $\gamma_1=\gamma_2$.
\begin{cor}
 Let $\Gamma$ act on $\Pp(H)$ in a way that is induced by a uniformly bounded representation of $\Gamma$ on $H$. Then $\Gamma$ never acts freely on the set $X_\pi$.\par 
Moreover, every element in $\Gamma$ fixes some point inside $X_\pi$.
\end{cor}
\begin{proof}\frei\\
Every $\gamma\in\Gamma$ generates an amenable subgroup (finite or $\Z$). Thus, there is a fixed point for this subgroup in $X_\pi$ (we apply Theorem \ref{existfixed} to $X_\pi$).
\end{proof}
Since actions on $\Pp(H)$ coming from linear representations are never free, it is natural to ask for possible stabilizers. The following theorem shows, that a group $\Gamma$ acting on a GCB space $X$ by bicombing respecting maps will either have a fixed point or all stabilizers are of infinite index.
\begin{thm}\label{nochgebraucht}
 Let $\Gamma$ act by bicombing-respecting maps on a GCB-space $X$ such that some finite index subgroup $\Lambda<\Gamma$ fixes a point in $X$. Then $\Gamma$ has a fixed point.
\end{thm}
\begin{proof}\frei\\
Let $\Lambda<\Gamma$ be a subgroup of index $n$ having a fixed point $x$ in $X$. Furthermore, let $\{e=\gamma _1,..,\gamma _n\}\subset \Gamma$ be a choice of representatives of the cosets $\Gamma/H$. \par 
Then $\Gamma=\dcup\limits_{i\in[n]}\gamma _iH$ and multiplying from the left with elements from the set $\{\gamma _i,i\in[n]\}$ or $H$ permutes the cosets $\gamma _iH$. In other words, multiplying with $\gamma \in \Gamma$ yields a bijection $\phi_\gamma :[n]\to[n]$ in such a way, that $\gamma \gamma _i\in \gamma _{\phi_\gamma (i)}H$.\par 
Define $y:=b(\{\gamma _i\cdot x|i\in[n]\})$. Then, for arbitrary $\gamma \in \Gamma$, we see, that for some $\lambda\in\Lambda$
\begin{align*}
 \gamma \cdot y&=\gamma \cdot b(\{\gamma _i\cdot x|i\in[n]\})=b(\{\gamma \gamma _i\cdot x|i\in[n]\})=b(\{\gamma _{\phi_\gamma (i)}\lambda\cdot x|i\in[n]\})=y
\end{align*}
and $y$ is a fixed point for the $\Gamma$-action.
\end{proof}
\begin{rem}
 In the theorem above, we did not assume Property (C) or boundedness.
\end{rem}
We can immediately conclude the following corollary, which shows in particular that virtually unitarisable groups are unitarisable.
\begin{cor}\label{virtunit}
 Let $\Gamma$ be a group containing a finite-index unitarisable subgroup $\Lambda$. Then $\Gamma$ is unitarisable and the constants in Theorems $\ref{eins}$ and $\ref{einsgeom}$ are at most the infimum over the corresponding constants coming from unitarisable finite index subgroups.
\end{cor}
\begin{proof}\frei\\
 Let $\pi:\Gamma\to\Aut(H)$ be a uniformly bounded representation of $\Gamma$ on some Hilbert space $H$, and $\rho_\pi$ the induced action on $\Pp(H)$.\par 
Then, by Corollary \ref{einsgeom}, there is a $\Lambda$-fixed point $x_\Lambda$, $C(\Lambda)+\frac{\alpha(\Lambda)}2\diam(\pi)$-close to the $\Lambda$-orbit of $\id_H\in\Pp(H)$. Hence it lies in the $C(\Lambda)+\frac{\alpha(\Lambda)}2\diam(\pi)$-neighbourhood of the $\Gamma$-orbit of $\id_H$.\par 
Now, Theorem \ref{nochgebraucht} yields a $\Gamma$-fixed point (proving, that $\Gamma$ is unitarisable), which (by construction) is the barycenter of the finite set of $\Gamma$-translates of $x_\Lambda$. Hence, it is at most $C(\Lambda)+\frac{\alpha(\Lambda)}2\diam(\pi)$ away from the $\Gamma$-orbit of $\id_H$ (and therefore, as it is a $\Gamma$-fixed point, from $\id_H$ itself). \par 
Thus, we have $\alpha(\Gamma)\leq\alpha(\Lambda)$ as well as $C(\Gamma)\leq C(\Lambda)$ (and therefore, $K(\Gamma)\leq K(\Lambda)$ for the universal constant $K$ in Theorem \ref{eins}).
\end{proof}
For the following corollary, the reader may be reminded, that an action of a group $\Gamma$ on a topological space $X$ is \textbf{proper}, if preimages of compact subsets of $X\times X$ under the map $\rho:\Gamma\times X\to X\times X,~(\gamma ,x)\mapsto(\gamma x,x)$ are compact.
\begin{cor}\label{yuio}
 Let $\Gamma$ be a discrete group acting properly on a Property (C) GCB-space $X$ by bicombing-respecting isometries and with at least one bounded orbit.\par 
Then, every amenable subgroup of $\Lambda$ is finite and $\Gamma$ is a torsion group.
\end{cor}
\begin{proof}\frei\\
$\Lambda<\Gamma$ be an amenable subgroup and $x\in X$ be a fixed point of $\Lambda$ (by Theorem \ref{existfixed}).\par 
Then, if $\Lambda$ is of infinite order, there is an infinite stabilizer for some $x\in X$ and the action cannot be proper. In particular, since every element generates an amenable subgroup, $\Gamma$ has to be a torsion group.
\end{proof}
\begin{thm}\label{extend}
Let $0\to \Gamma'\to \Gamma\to \Lambda\to 0$ be an extension of a unitarisable group $\Gamma'$ by an amenable group $\Lambda$. Then, $\Gamma$ is unitarisable and its universal constants (as defined in Theorem \ref{eins}) fullfill $K(\Gamma)\leq K(\Gamma')$ and $\alpha(\Gamma)\leq\alpha(\Gamma')+2$.
\end{thm}
\begin{proof}\frei\\
Let $\pi$ be a uniformly bounded representation of $\Gamma$ on $H$ and  let $x\mapsto \gamma x$ denote the induced action on $\Pp(H)$. Then the fixed point set $\Pp(H)^{\Gamma'}$ of the subgroup $\Gamma'$ is non-empty ($\Gamma'$ being unitarisable), closed (with respect to $\tau_d$ and $\tau_w$) and convex (both, linearly and metrically).\par 
Also, since $\Gamma'$ is normal in $\Gamma$, we have $\gamma_1^{-1}\gamma_2\gamma_1\in \Gamma'$ for all $\gamma_1\in \Gamma,~\forall \gamma_2\in \Gamma'$ and hence
\begin{align*}
\gamma_1\cdot (\gamma_2x)&=\left(\gamma_2\left(\gamma_2^{-1}\gamma_1\gamma_2\right)\gamma_2^{-1}\right)\cdot \gamma_2x=\gamma_2\cdot\left(\gamma_2^{-1}\gamma_1\gamma_2\right)\cdot x=\gamma_2x&\forall x\in\Pp(H)^{\Gamma'}
\end{align*}
proving, that $\Pp(H)^{\Gamma'}$ is a $\Gamma$-invariant Property-(C) GCB-space with trivial $\Gamma'$-action.\par 
Now, fix some $\tilde x\in\Pp(H)^{\Gamma'}$ minimizing the distance to $\id$. Then, we have
\begin{align*}
d(\tilde x,\id)\leq\ln K(\Gamma')+\frac{\alpha(\Gamma')}2\diam(\Gamma'\id)\leq\ln K(\Gamma')+\frac{\alpha(\Gamma')}2\diam(\Gamma\id)\end{align*}
and $\tilde x$ is in the $\left(\ln K(\Gamma')+\frac{\alpha(\Gamma')}2\diam(\Gamma\id)\right)$-neighbourhood of the $\Gamma$-orbit $\Gamma\cdot\id_H$. But then, this is also true for any image of $\tilde x$ under $\rho_\pi(\gamma)$.
\par Hence, we have $\Gamma/\Gamma'\cong\Lambda$ acting on $\Pp(H)^{\Gamma'}$ with bounded orbits and by Theorem \ref{existfixed}, we find a $\Lambda$-fixed point $\hat x$ convex close the orbit $\Gamma\tilde x$ in $\Pp(H)^{\Gamma'}$. In particular, $\hat x$ is fixed by the whole group $\Gamma$ and hence implies the unitarisability of $\pi$.\par 
By construction, $\hat x$ is a weak operator limit of points, which lie in the closed convex hull of the $\Gamma\tilde x$ which in turn had a distance of at most $\left(\ln K(\Gamma')+\frac{\alpha(\Gamma')}2\diam(\Gamma\id)\right)$ from $\Gamma\id$. \par
Hence, the sequence lies in the closed $\left(\ln K(\Gamma')+\frac{\alpha(\Gamma')}2\diam(\Gamma\id)\right)$-neighbourhood of the $\tau_w$-compact space $X_\pi$. Hence, it is itself $\tau_w$-compact. Therefore, the limit point $\hat x$ will be at most of distance $\left(\ln K(\Gamma')+\frac{\alpha(\Gamma')}2\diam(\Gamma\id)\right)$ to $X_\pi$ and therefore,
\begin{align*}
d(\hat x,\Gamma\id)\leq d(\hat x,X_\pi)+\diam(\Gamma\id)\leq\left(\ln K(\Gamma')+\frac{\alpha(\Gamma')+2}2\diam(\Gamma\id)\right)
\end{align*} 
\end{proof}
\begin{rem}\label{end}
Observe, that by moving from some unitarisable group $\Gamma'$ to a group $\Gamma$, which is an extension of $\Gamma'$ by some amenable group or contains $\Gamma'$ as a finite index subgroup, we don't change its "distance from being amenable" in the following way:\par 
Remember, that a group $G$ is amenable, if and only if for any $G$-action on $\Pp(H)$ coming from some representation, we can find a fixed point in $X_\pi$. In this way, we can say, that a unitarisable group $G$ is $\delta$ away from being amenable ($\delta$ being a linear function of the size of the representation), if we always find a fixed point, which has distance at most $\delta$ from $X_\pi$.\\[0.3cm]
Now, as we see from the proofs above, even though the constant $\alpha(\Gamma)$ might be different from $\alpha(\Gamma')$, we always find a fixed point in the $\left(\ln K(\Gamma')+\frac{\alpha(\Gamma')}2\diam(\Gamma\id)\right)$-neighbourhood of $X_\pi$.
\end{rem}

\end{document}